\documentclass[a4paper, 12pt]{amsart}

\def\Time-stamp%
: <#1 #2 #3>{#1 #2}

\usepackage{braket}
\usepackage{tabmac}
\newcommand{\Tableau}[2][sY]{{\text{\footnotesize
\scriptsize
\tableau[#1]{#2}}}}
\usepackage{amsmath, amssymb,  url}

\usepackage{amsthm}

\theoremstyle{plain}   
\newtheorem{thm}{Theorem}[section]
\newtheorem{theorem}[thm]{Theorem}

\newtheorem{lemma}[thm]{Lemma}

\theoremstyle{remark}  
\newtheorem{remark}[thm]{Remark}

\newtheorem{example}[thm]{Example}

\theoremstyle{definition}  
\newtheorem{definition}[thm]{Definition}
\newtheorem{algorithm}[thm]{Algorithm}

\newcommand{\YY}{\mathbb{Y}}
\newcommand{\numof}[1]{|#1|}

\newcommand{\PPP}{\mathcal{P}}
\newcommand{\QQQ}{\mathcal{Q}}
\newcommand{\RRR}{\mathfrak{R}}
\newcommand{\SSS}{\mathcal{S}}

\newcommand{\disjointunion}{\amalg}

\author{NUMATA, Yasuhide}
\address{Department of Mathematical Sciences, Shinshu
University, 3-1-1 Asahi, Matsumoto, Nagano 390-8621 JAPAN}
\email{nu@math.shinshu-u.ac.jp}
 \thanks{Supported by JSPS KAKENHI Grant Number 00455685.}
\title[A bijective proof of the Cauchy identity]{A bijective proof of the Cauchy identity for Grothendieck polynomials}
\begin{document}

\begin{abstract}
We consider 
pairs of a set-valued column-strict tableau and
a reverse plane partition of the same shape.
We introduce algortithms for them,
 which implies
a bijective proof for
the finite sum Cauchy identity 
for Grothendieck polynomials and dual Grothendieck polynomials.
\end{abstract}
\maketitle

\maketitle

\section{Introduction}
In this paper,
 we consider the stable Grothendieck polynomials $G_\lambda$
and the dual Grothendieck polynomials $g_\lambda$.
They are $K$-theoretic analogues of Schur functions $s_\lambda$.
Grothendieck polynomials introduced in Lascoux--Sch\"utzenberger
 \cite{MR686357}
are the representatives of $K$-theory classes of structure 
sheaves of Schubert varieties.
Fomin and Kirillov introduced
the stable Grothendieck polynomial 
in their combinatorial study of the Grothendieck polynomials
 \cite{MR1394950}.
The stable Grothendieck polynomial is 
a limit of the Grothendieck polynomial.
Buch \cite{MR1946917} showed
stable Grothendieck polynomials 
can be written as weighted generating functions of 
a kind of tableaux.
The dual Grothendieck polynomials is a dual basis of the stable
 Grothendieck polynomial,
and can be written as weighted generating functions of 
another kind of tableaux. 
See also 
\cite{MR3007174},
\cite{MR1932326} and
\cite{MR2377012}.
Schur functions $s_\lambda$ satisfy 
the Cauchy identity
\begin{align*}
 \sum_{\lambda}s_\lambda(x)\cdot s_\lambda(y)=\prod_{i,j}\frac{1}{1-x_iy_j}.
\end{align*}
In \cite{MR2377012},
Lam and Pylyavskyy showed
an analogue of the classical Cauchy identity 
for products of the stable
Grothendieck polynomial and the dual Grothendieck polynomial.
In \cite{MR3249830}, 
Lascoux and Naruse generalized the identity as the finite sum Cauchy identity
\begin{align*}
 \sum_{\lambda}G_\lambda(x_1,\ldots,x_n)\cdot g_\lambda(y_1,\ldots,y_n)
&=
 \sum_{\lambda}s_\lambda(x_1,\ldots,x_n)\cdot s_\lambda(y_1,\ldots,y_n)\\
&=\prod_{i=1}^n\prod_{j=1}^n\frac{1}{1-x_iy_j},
\end{align*}
where the sum is over all $\lambda$ containing by the rectangle shape $(w^n)$.
Since the polynomials $s_\lambda$, $G_\lambda$ and $g_\lambda$ are
weighted generating functions of tableaux,
the finite sum Cauchy identity
means
the existence of
a weight-preserving bijection between some kinds of tableaux.
The main purpose of this paper
is to give a bijective proof
of the finite sum Cauchy identity
for stable Grothendieck polynomials and dual Grothendieck polynomials.

This paper is organized as follows:
In Section \ref{section:tableau},
we define some notion for tableaux
and define the weighted generating functions them.
We introduce 
some algorithms to define
bijections in Section \ref{sec:defofalg}.
We use the bumping and the sliding
to define our algorithm.
In Section \ref{sec:main},
we give the main theorem which implies the Cauchy identity.
 We show the main theorem  in Section \ref{sec:proof}.
In Section \ref{sec:finalrmk},
we consider the case of Young diagram 
with one column.

\section{Three kinds of tableaux}
\label{section:tableau}
Here we define some notation for Young diagrams and tableaux.
Let $\YY$ be the set of Young diagrams. 
We write $(\lambda_1,\lambda_2,\ldots)$
to denote the  Young diagram
consisting of $\lambda_i$ boxes in the $i$-th row.
We use so-called English notation.
For $\lambda\in\YY$,
we call a box $(i,j)\in \lambda$
a \emph{corner} if 
 $(i+1,j)\not\in \lambda$ 
and
 $(i,j+1)\not\in \lambda$.
We call $(i,j)\not\in\lambda$
a \emph{cocorner}
if
 $(i-1,j)\in \lambda$ 
and
 $(i,j-1)\in \lambda$.
For $\lambda'$, $\lambda\in \YY$ satisfying $\lambda'\subset\lambda$,
we define the \emph{skew Young diagram} $\lambda/\lambda'$
to be the subset $\lambda\setminus\lambda'$ of boxes.
For $\lambda\in\YY$,
$\left|\lambda\right|$ denotes the number of the boxes in $\lambda$.
We call $T$ an \emph{integer-valued column-strict tableau} of shape $\lambda$
if
\begin{enumerate}
\item $T_{i,j}$ is a positive integers;
\item if $(i,j),(i',j) \in \lambda$ satisfy $i<i'$,
 then $T_{i,j}<T_{i',j}$; and
\item if $(i,j),(i,j') \in \lambda$ satisfy $j<j'$,
 then $T_{i,j}\leq T_{i,j'}$.
\end{enumerate} 
We sometimes call integers in a tableau \emph{alphabets},
to distinguish between them and the other integers, e.g., coordinates of boxes.
For a skew Young diagram $\lambda/\lambda'$,
we define an
\emph{integer-valued column-strict skew tableau} of shape
$\lambda/\lambda'$
in the same manner.
We define $\SSS_\lambda$ to be the set of 
integer-valued column-strict tableaux of shape $\lambda$.
For $T\in\SSS_\lambda$,
define $\mu(T)$ to be the sequence $(\mu_1,\mu_2,\ldots)$
such that $\mu_k=\numof{\Set{(i,j)\in\lambda|T_{i,j} = k }}$.
In other words,
$\mu_k$ is the multiplicity of the alphabet $k$
in the tableau $T$.
For a Young diagram $\lambda\in\YY$, and formal variables $x=(x_1,x_2,\ldots)$,
we define the weighted generating function $s_\lambda(x)$ by
\begin{align*}
 s_\lambda(x)=\sum_{T\in\SSS_\lambda} x^{\mu(T)},
\end{align*}
where $x^{\alpha}$ stands for the monomial
$x_1^{\alpha_1}x_2^{\alpha_2}\cdots$.
The generating functions $s_\lambda(x)$ are called 
\emph{Schur functions}.
For the Young diagram $(l)$ consisting of one row with $l$ boxes,
$s_{(l)}(x)$ equals the $l$-th homogeneous complete symmetric polynomial $h_l(x)$.
For the Young diagram $(1^l)$ consisting of one column with $l$ boxes,
$s_{(1^l)}(x)$ equals the $l$-th elementary symmetric polynomial $e_l(x)$.

For $\lambda\in\YY$,
we call $P$ a \emph{set-valued column-strict tableau} 
of shape $\lambda$
if
\begin{enumerate}
\item $P_{i,j}$ is a nonempty finite subset of positive integers;
\item if $(i,j),(i',j) \in \lambda$ satisfy $i<i'$,
 then $\max(P_{i,j})<\min(P_{i',j})$; and
\item if $(i,j),(i,j') \in \lambda$ satisfy $j<j'$,
 then $\max(P_{i,j})\leq\min(P_{i,j'})$.
\end{enumerate} 
We define $\PPP_\lambda$ to be the set of 
set-valued column-strict tableaux of shape $\lambda$.
For $P\in \PPP_\lambda$,
we define $\varepsilon(P)$ by
\begin{align*}
\varepsilon(P)&=\sum_{(i,j)\in\lambda} (\numof{P_{i,j}}-1).
\end{align*}
A set-valued column-strict tableau $P\in\PPP_\lambda$ satisfies $\varepsilon(P)=0$
if and only if $\numof{P_{i,j}}=1$ for all $(i,j)\in\lambda$.
We can identify $P\in\PPP_\lambda$
satisfying $\varepsilon(P)=0$
with an integer-valued column-strict tableau.
For $\lambda\in\YY$,
we define $\tilde\SSS_\lambda$ and $\PPP_\lambda^+$ by
\begin{align*}
\tilde\SSS_\lambda&=\Set{P\in\PPP_\lambda|\varepsilon(P)=0}, \\
\PPP_\lambda^+&=\Set{P\in\PPP_\lambda|\varepsilon(P)>0} \\
 &=\PPP_\lambda\setminus\tilde\SSS_\lambda.
\end{align*}
We define $\mu(P)$ to be the sequence of integers $(\mu_1,\mu_2,\ldots)$
such that $\mu_k=\numof{\Set{ (i,j) \in \lambda| k\in P_{i,j}}}$.
It is easy to show that
\begin{align}
 \left|\lambda\right| =-\varepsilon(P)+\sum_i \mu_i.
\label{eq:sum:mu}
\end{align}
For a Young diagram $\lambda\in \YY$ and formal variables $\beta$ and $x=(x_1,x_2,\ldots)$,
we define a weighted generating function $G_\lambda(\beta,x)$
by
\begin{align*}
 G_\lambda(\beta,x)
&=\sum_{P\in\PPP_\lambda} (-\beta)^{\varepsilon(P)}x^{\mu(P)}.
\end{align*}
We call $G_\lambda(x)=G_\lambda(1,x)$ a \emph{(stable Gra\ss{}mannian) Grothendieck polynomial}.
It is easy to show that $G_\lambda(0,x)=s_\lambda(x)$.

For $\lambda\in\YY$,
we call $Q$ a \emph{reverse plane partition} of $\lambda$
if
\begin{enumerate}
\item $Q_{i,j}$ is a positive integer;
\item if $(i,j),(i',j) \in \lambda$ satisfy $i<i'$,
 then $Q_{i,j}\leq Q_{i',j}$; and
\item if $(i,j),(i,j') \in \lambda$ satisfy $j<j'$,
 then $Q_{i,j}\leq Q_{i,j'}$.
\end{enumerate} 
We define $\QQQ_\lambda$ to be the set of 
reverse plane partitions of $\lambda$.
For $Q\in\QQQ_\lambda$,
we define $\delta(Q)$ by
\begin{align*}
 \delta(Q)&=\numof{\Set{(i,j)\in\lambda| Q_{i,j}=Q_{i+1,j}}}. 
\end{align*}
For $Q\in\QQQ_\lambda$,
it follows from definition 
that $Q\in\SSS_\lambda$
if and only if
$\delta(Q)=0$.
For $\lambda\in\YY$, 
we define $\QQQ_\lambda^+$ by
\begin{align*}
\QQQ_\lambda^+&=\Set{Q\in \QQQ_\lambda | \delta(Q)>0}\\
 &=\QQQ_\lambda\setminus\SSS_\lambda.
\end{align*}
We also define $\nu(Q)$ to be the sequence $(\nu_1,\nu_2,\ldots)$
such that
\begin{align*}
\nu_k=\numof{\Set{j|\text{There exists $(i,j)\in\lambda$ such that $Q_{i,j} = k $}}}. 
\end{align*}
It is easy to show that
\begin{align}
\left|\lambda\right| =\delta(Q)+\sum_i \nu_i.
 \label{eq:nu}
\end{align}
For a Young diagram $\lambda\in \YY$ and formal variables $\beta$ and $y=(y_1,y_2,\ldots)$,
we define a weighted generating function $g_\lambda(\beta,y)$
by
\begin{align*}
 g_\lambda(\beta,y)
&=\sum_{Q\in\QQQ_\lambda} \beta^{\delta(Q)}y^{\nu(Q)}.
\end{align*}
We call $g_\lambda(y)=g_\lambda(1,y)$
a \emph{dual Grothendieck polynomial}.
It is easy show that $ g_\lambda(0,y)=s_\lambda(y)$.

\section{Our algorithm}
\label{sec:defofalg}
\subsection{The row insertions and the jeu de taquin}
\label{sec:knownalgorithm}
Here we recall some well-know algorithms for 
integer-valued column-strict tableaux.
\begin{algorithm}[Row bumping]
 Let $T_{i,1},\ldots,T_{i,\lambda_i}$ 
be the $i$-th row of an integer-valued column-strict tableau $T\in \SSS_\lambda$,
and $x$ an alphabet.
We define $T'$ and $y$ in the following manner:
\begin{enumerate}
 \item  Find $j$ satisfying the following:
 \begin{enumerate}
  \item if $j'< j$ and $x'=T_{i,j'}$, then
	$x'\leq x$; and
  \item if $j\leq j''$ and  $x''=T_{i,j''}$, then
	$x < x''$.
 \end{enumerate}
 \item If $T$ has a box at $(i,j)$, then
       define $y$ to be $T_{i,j}$ and
       define $T'$ to be the tableau obtained from $T$
       by putting the alphabet $x$ at the box $(i,j)$.
 \item If $T$ has no box at $(i,j)$, then
       define $y$ to be null
       and define $T'$ to be the tableau obtained from $T$
       by adding the new box $(i,j)$ with the alphabet $x$.
\end{enumerate}
\end{algorithm}

If we 
repeat inserting the integer bumped out from the $i$-th row 
into the $(i+1)$-th row by the row bumping
until the null is bumped out,
then we obtain an integer-valued column-strict tableau.
We call the whole process  the \emph{row-insertion}.
The row-insertion stops when
an integer is inserted into a corner as a new box.
We call the set of boxes where an integer is inserted 
in the process of the row-insertion
the \emph{bumping route}.
If an integer is inserted into the box $(1,j_1)$ by the first bumping of 
the row-insertion,
then the bumping route is a subset
of $\Set{(i,j)|j\leq j_1}$.

\begin{algorithm}[Reverse row bumping]
 Let $T_{i,1},\ldots,T_{i,\lambda_i}$ be the $i$-th row of 
an integer-valued column-strict tableau $T\in \SSS_\lambda$,
and $y$ an alphabet.
We define $T'$ and $x$ in the following manner:
\begin{enumerate}
 \item  Find $j$ satisfying the following:
 \begin{enumerate}
  \item if $j'\leq j$ and $y'=T_{i,j'}$, then
	$y'<y$; and
  \item if $j<j''$ and  $y''=T_{i,j''}$, then
	$y \leq y''$.
 \end{enumerate}
 \item 
       Define $x$ to be $T_{i,j}$ and
       define $T'$ to be the tableau obtained from $T$
       by putting the alphabet $y$ at the box $(i,j)$.
\end{enumerate}
\end{algorithm}

Let $T$ be an integer-valued column-strict tableau of shape $\lambda$,
 $(i_1,j_1)$ a corner of $\lambda$.
If we remove a corner  $(i_1,j_1)$ from $T$ and 
insert $T_{i_1,j_1}$ into the $(i_1-1)$-th row of $T$
by the reverse row-bumping,
then we obtain an integer-valued column-strict tableau
and an integer bumped out from the $(i_1-1)$-th row.
If we 
repeat inserting the integer bumped out from the $(i+1)$-th row
into the $i$-th row 
by the reverse row bumping,
then we obtain an integer-valued column-strict tableau 
and an integer bumped out.
We call the whole process  the \emph{reverse row-insertion}.
The reverse row-insertion is the inverse of the row-insertion.
We call the set of boxes where an integer is bumped out in the process
of the reverse row-insertion the \emph{reverse-bumping route}.

Next we recall another kind of algorithms.
\begin{algorithm}[Sliding]
 Let $T^{(0)}$ be an integer-valued column-strict tableau 
 with a null box $(i_0,j_0)$.
 We define $T$ in the following manner:
\begin{enumerate}
 \item If $T^{(0)}_{i_0+1,j_0}>T^{(0)}_{i_0,j_0+1}$ or $T^{(0)}$ has no box at $(i_0+1,j_0)$,
       then let $T$ be the tableau obtained from $T^{(0)}$ by swapping 
       the entries of boxes $(i_0,j_0)$ and $(i_0,j_0+1)$.
 \item If $T^{(0)}_{i_0+1,j_0}\leq T^{(0)}_{i_0,j_0+1}$ or $T^{(0)}$ has no box at $(i_0,j_0+1)$,
       then let $T$ be the tableau obtained from $T^{(0)}$ by swapping 
       the entries of boxes $(i_0,j_0)$ and $(i_0+1,j_0)$.
\end{enumerate}
\end{algorithm}

Let $T$ be an integer-valued column-strict skew tableau
of shape $\lambda/(1)$.
We regard the box $(1,1)$ as a null box.
Repeat the sliding until the null box reaches a corner of $\lambda$,
and then remove the null box.
Then we obtain an integer-valued column-strict tableau.
We call the whole process the \emph{jeu de taquin}.
We call the boxes where the null box passed 
in the process of the jeu de taquin
the \emph{sliding route}.


\begin{algorithm}[Reverse sliding]
 Let $T^{(1)}$ be an integer-valued column-strict tableau 
 with a null box $(i_1,j_1)$.
 We define $T$ in the following manner:
\begin{enumerate}
 \item If $i_1=1$ or $T^{(1)}_{i_1-1,j_1}<T^{(1)}_{i_1,j_1-1}$,
       then let $T$ be the tableau obtained from $T^{(1)}$ by swapping 
       the entries of boxes $(i_1,j_1)$ and $(i_1,j_1-1)$.
 \item If  $j_1=1$ or $T^{(1)}_{i_1-1,j_1}\geq T^{(1)}_{i_1,j_1-1}$,
       then let $T$ be the tableau obtained from $T^{(1)}$ by swapping 
       the entries boxes $(i_1,j_1)$ and $(i_1-1,j_1)$.
\end{enumerate}
\end{algorithm}
Let $T$ be an integer-valued column-strict tableau of shape $\lambda$,
and $(i,\lambda_{i}+1)$ a cocorner of $\lambda$.
Append a null box to the cocorner $(i,\lambda_{i}+1)$ of $T$.
Repeat the reverse sliding until the null box reaches $(1,1)$.
Then we obtain 
an integer-valued column-strict skew tableau
of shape $(\lambda_1,\ldots,\lambda_{i-1},\lambda_i+1,\lambda_{i+1},\ldots)/(1,1)$.
We call the whole process the \emph{reverse jeu de taquin}.
The reverse jeu de taquin is the inverse of the jeu de taquin.
We call the boxes where the null box passed 
in the process of the reverse jeu de taquin
the \emph{reverse-sliding route}.

\subsection{Operation for tableaux}
\label{subsec:alg:tab}
If $P\in \PPP_\lambda^+$,
then $P$ has a box containing more than one alphabets.
We define notation for the end of such boxes.
\begin{definition}
For $P\in\PPP_\lambda^+$
define $r_1(P)$ and  $r_1'(P)$  by
\begin{align*}
r_1(P)&=\max \Set{i| \text{ $\numof{P_{i,j}}>1$ for some $j$}},\\
r_1'(P)&=\max\Set{j| \numof{P_{r_1(P),j}}>0}.
\end{align*}
For $P\in \tilde\SSS_\lambda$,
we define $r_1(P)=0$.
\end{definition}
We can regard  $P\in \PPP^+_\lambda$ 
as an integer-valued column-strict (skew) tableau,
if we consider only 
the set 
\begin{align*}
\Set{(i,j)\in\lambda| i>r_1(P)}
\end{align*}
of boxes. 
Hence we define $r_2(P)$, $r_2'(P)$ and  $R(P)$
by the row bumping algorithm for integer-valued column-strict tableaux.
\begin{algorithm}
\label{algorithm:R(P)}
For $P\in\PPP_\lambda^+$,
we define 
the cocorner $(r_2(P),r_2'(P))$ of $\lambda$
and the set-valued column-strict tableau $R(P)$ 
in the following manner:
\begin{enumerate}
 \item Let $x=\max(P_{r_1(P),r_1'(P)})$.
 \item Remove the alphabet $x$ from the box $(r_1(P),r_1'(P))$ of $P$. 
 \item  Insert the alphabet $x$
	into the tableau 
	by the row insertion from the $(r_1(P)+1)$-th row.
	As the result of insertion,
	we obtain a new tableau $P^{(1)}$ of shape $\lambda^{(1)}$.
 \item Define $R(P)\in\PPP_{\lambda^{(1)}}$ to be $P^{(1)}$. 
 \item Define $(r_2(P),r_2'(P))$ to be the new box,
       i.e., the box in $\lambda^{(1)}$ but not in $\lambda$.
\end{enumerate}
\end{algorithm}
\begin{example}
\label{example:P}
Define set-valued column-strict tableaux $P$ and $P'$ by
\begin{align*}
 P&=
\Tableau[Y]{1&12&2&34&4&6 \\ 23&3&\tf 34&5&6&8\\3&4&5&6&7\\4&6&7&8 \\ 6 },\\
 P'&=
 \Tableau[Y]{1&12&2&34&4&6 \\ 23&3&34&5&\tf 67&8\\3&4&5&6&8\\4&6&7&8 \\ 6 }.
\end{align*}
In this case, $(r_1(P),r_1'(P))=(2,3)$ and $(r_1(P'),r_1'(P'))=(2,5)$.
To calculate 
$R(P)$ (\emph{resp.}\ $R(P')$)
we insert the alphabet $4=\max\set{3,4}$ 
(\emph{resp.}\ $7=\max\set{6,7}$)
from the third row of $P$ by
 the row-bumping. Hence we obtain
\begin{align*}
 R(P)&=
\Tableau[Y]{1&12&2&34&4&6 \\ 23&3&\tf
 3&5&6&8\\3&4&\textbf{4}&6&7\\4&\textbf{5}&7&8 \\ 6 &\tf \textbf{6}},\\
 R(P')&=
 \Tableau[Y]{1&12&2&34&4&6 \\ 23&3&34&5&\tf 6&8\\3&4&5&6&\textbf{7}\\4&6&7&8&\tf\textbf{8} \\ 6 }.
\end{align*}
Therefore $(r_2(P),r_2'(P))=(5,2)$ and $(r_2(P'),r_2'(P'))=(4,5)$.
\end{example}

A reverse plane partition $Q\in \QQQ_\lambda^+$ 
contains a box $(i,j)$ such that $Q_{i,j}=Q_{i+1,j}$.
We define notation for the end of such boxes.
\begin{definition}
For $Q\in \QQQ_\lambda^+$,
we define $f_1(Q)$ and $f_1'(Q)$ by
\begin{align*}
 f_1(Q)&=\max \Set{i|\text{$Q_{i,j}=Q_{i+1,j}$ for some $j$}},\\
 f_1'(Q)&=\max\Set{j| Q_{f_1(Q),j}=Q_{f_1(Q)+1,j}}.
\end{align*}
For $Q\in\SSS_\lambda$,
we define $f_1(Q)=0$.
\end{definition}
We can regard $Q\in \QQQ_\lambda^+$
as an integer-valued column-strict tableau
if 
we consider only the set
\begin{align*}
\Set{(i,j)\in\lambda| \text{$i\geq f_1(Q)$ and $j\geq f_1'(Q)$}} 
\setminus \Set{(f_1(Q),f_1'(Q))}
\end{align*}
of boxes. 
Hence we define $f_2(Q)$, $f_2'(Q)$ and $F(Q)$
by the jeu de taquin for 
integer-valued column-strict tableaux.
\begin{algorithm}
\label{algorithm:F}
For $Q\in \QQQ^+_\lambda$,
we define 
the corner $(f_2(Q),f_2'(Q))$ of $\lambda$
and the reverse plane partition $F(Q)$ in the following manner:
\begin{enumerate}
 \item Let the box $(f_1(Q),f_1'(Q))$ in $Q$ be null.
 \item Slide the null box to outside by the jeu de taquin.
 \item As the result of sliding,
       we obtain a new tableau $Q^{(1)}$ of shape $\lambda^{(1)}$.
 \item Define $F(Q)\in \QQQ_{\lambda^{(1)}}$ to be $Q^{(1)}$.
 \item Define $(f_2(P),f_2'(P))$ to be 
       the removed box
       $\lambda/ \lambda^{(1)}$.
\end{enumerate}
\end{algorithm}
\begin{example}
\label{example:Q}
Define reverse plane partition $Q$ by
\begin{align*}
 Q&=
\Tableau[Y]{1&1&2&4&5&5 \\ 1&\tf 2&3&4&4&6\\2&2&6&6&7\\3&4&7&9 \\ 5 }.
\end{align*}
In this case, $(f_1(Q),f_1'(Q))=(2,2)$.
To calculate 
$F(Q)$, 
we put a null box at $(2,2)$.
Then we obtain the following by the jeu de taquin:
\begin{align*}
 F(Q)&=
\Tableau[Y]{1&1&2&4&5&5 \\ 1&\tf \textbf{2}&3&4&4&6\\2&\textbf{4}&6&6&7\\3&\textbf{7}&\textbf{9}&\tf \\ 5 }.
\end{align*}
Therefore $(f_2(Q),f_2'(Q))=(4,4)$.
\end{example}

\subsection{Operation for pairs of tableaux}
\label{subsec:alg:pair}
For $\lambda\in\YY$,
we define
 $\RRR_\lambda$,  $\RRR^0_\lambda$ and 
$\RRR^+_\lambda$ 
by
\begin{align*}
 \RRR_\lambda&=\PPP_\lambda\times\QQQ_\lambda,\\
 \RRR^0_\lambda&=\tilde\SSS_\lambda\times\SSS_\lambda\\
 &=\Set{(P,Q)\in\RRR_\lambda|\varepsilon(P)=\delta(Q)=0},\\
 \RRR^+_\lambda&=\RRR_\lambda\setminus \RRR^0_\lambda\\
 &=\Set{(P,Q)\in\RRR_\lambda|\varepsilon(P)+\delta(Q)>0}.
\end{align*}
We decompose $\RRR^+_\lambda$ into the two subsets
${\check{\RRR}}_\lambda$ and ${\hat{\RRR}}_\lambda$.
\begin{definition}
Define ${\check{\RRR}}_\lambda$ to be 
the set of pairs $(P,Q)$ of tableaux in $\RRR^+_\lambda$
satisfying
\begin{enumerate}
\item $r_1(P)>f_1(Q)$; or
\item $r_1(P)=f_1(Q)$ and $r_2(P)\leq f_2(Q)$.
\end{enumerate}
We also 
define ${\hat{\RRR}}_\lambda$ to be 
the set of pairs $(P,Q)$ of tableaux in $\RRR^+_\lambda$
satisfying
\begin{enumerate}
\item $r_1(P)<f_1(Q)$; or
\item $r_1(P)=f_1(Q)$ and $r_2(P)> f_2(Q)$.
\end{enumerate}
\end{definition}
It is easy to show that 
$\RRR^+_\lambda={\check{\RRR}}_\lambda\disjointunion {\hat{\RRR}}_\lambda$.
First we define a map $\check{\varphi}$ from ${\check{\RRR}}_\lambda$ to $\QQQ^+$.
\begin{algorithm}
\label{algorithm:checkvarphi}
 For a $(P,Q)\in {\check{\RRR}}_\lambda$,
 we define $\check{\varphi}(P,Q) \in\QQQ^+$ in the following manner:
\begin{enumerate}
 \item Let $Q^{(1)}$ be $Q$.
 \item Let $Q^{(2)}$ be the tableau obtained from $Q$ by adding 
       a null box at the cocorner $(r_2(P),r_2'(P))$.
 \item 
       \label{def:E2:alg:step:jdt}
       Slide the null box $(r_2(P),r_2'(P))$ of $Q^{(2)}$
       into the inside of $Q^{(2)}$
       by the reverse jeu de taquin
       until the null box moves to $r_1(P)$-th row.
       Let $Q^{(3)}$ be the resulting tableau.
 \item 
       Define $(r_3(P,Q),r_3'(P,Q))$  to be the null box  of $Q^{(3)}$. 
 \item 
       \label{def:E2:alg:step:putalphabet}
       Define $\check{\varphi}(P,Q)$ to be the tableau obtained from  $Q^{(3)}$ 
       by putting the alphabet $Q^{(1)}_{r_3(P,Q),r_3'(P,Q)}$
       into the null box $(r_3(P,Q),r_3'(P,Q))$ of  $Q^{(3)}$. 
\end{enumerate}
\end{algorithm}

\begin{remark}
In the case where $r_1(P)>f_1(Q)$,
$Q$ is column-strict 
if we consider only boxes from $r_1(P)$-th row to $r_2(P)$-th row.
Hence Step \ref{def:E2:alg:step:jdt} is well-defined.
In the case where  $r_1(P)=f_1(Q)$ and $r_2(P)\leq f_2(Q)$,
$Q$ is column-strict 
if we consider only the boxes 
\begin{align*}
\Set{(i,j)\in\lambda| \text{$i\geq f_1(Q)$ and $j\geq f_1'(Q)$}} \setminus \Set{(f_1(Q),f_1'(Q))}.
\end{align*}
Since $r_2(P)\leq f_2(Q)$,
it follows from Lemma \ref{lemma:slidingroute1} 
that
the null box reaches a box in $\Set{(r_1(P),j)|j>f_2(Q)}$.
Hence Step \ref{def:E2:alg:step:jdt} is also well-defined in this case.
\end{remark}
\begin{remark}
\label{remark:varphi}
The reverse jeu de taquin in Step \ref{def:E2:alg:step:jdt}
 stops if the null box arrives at the $r_1(P)$-th row.
 Hence $Q^{(3)}_{r_3(P,Q)+1,r_3'(P,Q)}=Q^{(1)}_{r_3(P,Q),r_3'(P,Q)}$.
 Moreover, 
 if $(r_3(P,Q),r_3'(P,Q)+1)\in\lambda$,
 then 
\begin{align*}
Q^{(3)}_{r_3(P,Q),r_3'(P,Q)+1}=Q^{(1)}_{r_3(P,Q),r_3'(P,Q)+1}\geq Q^{(1)}_{r_3(P,Q),r_3'(P,Q)}. 
\end{align*}
 Since we put the alphabet $Q^{(1)}_{r_3(P,Q),r_3'(P,Q)}$
 on the box $(r_3(P,Q),r_3'(P,Q))$ of $Q^{(3)}$
 in Step \ref{def:E2:alg:step:putalphabet},
 $\check{\varphi}(P,Q)$ is in $\QQQ^+$
and the shape of $\check{\varphi}(P,Q)$
is $\lambda\cup \Set{(r_2(P),r_2'(P))}$. 
\end{remark}
\begin{example}
Consider $P'$ and $Q$ in Examples \ref{example:P} and \ref{example:Q}.
Since $r_1(P')=f_1(Q)=2$ and
$r_2(P')=f_2(Q)=4$,
the pair $(P',Q)$ is in ${\check{\RRR}}_\lambda$.
To calculate $\check{\varphi}(P',Q)$,
we add a null box 
to the cocorner $(r_2(P'),r_2'(P'))=(4,5)$ of $Q$.
Then we obtain 
the following tableau by the reverse jeu de taquin:
\begin{align*}
\Tableau[Y]{1&1&2&4&5&5 \\ 1&2&\tf&4&4&6\\2&2&\textbf{3}&6&7\\3&4&\textbf{6}&\textbf{7}&\tf \textbf{9}\\ 5 }.
\end{align*}
Finally we put the alphabet $3$,
which lies bellow the null box, 
into the null box to obtain
\begin{align*}
 \check{\varphi}(P',Q)=\Tableau[Y]{1&1&2&4&5&5 \\ 1& 2&\tf \textbf{3} &4&4&6\\2&2&\textbf{3}&6&7\\3&4&\textbf{6}&\textbf{7}&\tf \textbf{9}\\ 5 }.
\end{align*}
\end{example}
\begin{definition}
For $(P,Q) \in {\check{\RRR}}_{\lambda}$,
we define $\check{\Phi}(P,Q)$ by
$\check{\Phi}(P,Q)=(R(P),\check{\varphi}(P,Q))$.
\end{definition}

Next we define a map $\hat{\varphi}$ from ${\hat{\RRR}}_\lambda$ to $\PPP^+$.
\begin{algorithm}
\label{algorithm:hatvarphi}
 For $(P,Q)\in{\hat{\RRR}}_\lambda$,
we define $\hat{\varphi}(P,Q)\in\PPP^+$ 
in the following manner:
\begin{enumerate}
\item 
      \label{def:F2:alg:step:pickupuniqelm}
      Let $x$ be the alphabet in $P_{f_2(Q),f_2'(Q)}$.
\item Let $P^{(2)}$ be the tableau obtained from $P$
      by removing the corner $(f_2(Q),f_2'(Q))$.
\item 
      \label{def:F2:alg:step:revbumping}
      Insert the alphabet $x$ into $P^{(2)}$
      by the reverse row insertion from the $(f_2(Q)-1)$-th row
      until some number $y$ is bumped out from the $(f_1(Q)+1)$-th row.
      Let $P^{(3)}$ be the resulting tableau.
\item 
      \label{def:F2:alg:step:findpos}
      Find $j$ satisfying the following:
      \begin{enumerate}
       \item if $j'\leq j$ and $y' \in P^{(3)}_{f_1(Q),j'}$, then
	     $y'<y$; and
       \item if $j<j''$ and  $y'' \in P^{(3)}_{f_1(Q),j''}$, then
	     $y \leq y''$.
      \end{enumerate}
\item Define $(f_3(P,Q),f_3'(P,Q))$ to be the box $(f_1(Q),j)$.
\item Define $\hat{\varphi}(P,Q)$ to be the tableau
      obtained from $P^{(3)}$
      by appending the alphabet $y$ into the box $(f_3(P,Q),f_3'(P,Q))$.
\end{enumerate}
\end{algorithm}
\begin{remark}
For $(P,Q)\in {\hat{\RRR}}_\lambda$,
 the box $(f_2(Q),f_2'(Q))$ is a corner of $\lambda$
 and $\numof{P_{f_2(Q),f_2'(Q)}}$ equals one.
Hence $x$ in Step \ref{def:F2:alg:step:pickupuniqelm}
is unique.
\end{remark}
\begin{remark}
 Since $r_1(P)\leq f_1(Q)$,
 we can regard  $P$ as an integer-valued column-strict tableau
 if we consider only the boxes from 
 the $(f_1(Q)+1)$-th row to the $(f_2(Q)-1)$-th row.
Hence Step \ref{def:F2:alg:step:revbumping} 
is well-defined.
\end{remark}
\begin{remark}
\label{remark:hatphi}
Roughly speaking,
 $(f_3(P,Q),f_3'(P,Q))=(f_1(Q),j)$ is the unique box in the $f_1(Q)$-th row
such that
\begin{enumerate}
 \item  $y$ will be the maximum in the box; and 
 \item  the resulting tableaux will be column-strict if we insert $y$ into the box.
\end{enumerate}
If $\numof{P_{f_1(Q),j'}}=1$ for all $j'$,
then we can find  $j$ in Step \ref{def:F2:alg:step:findpos}
similarly to the case of the reverse bumping algorithm 
for integer-valued column-strict tableaux.
In the case where
$r_1(P)<f_1(Q)$,
$\numof{P_{f_1(Q),j'}}=1$ for all $j'$.
Hence we can find $j$ in Step \ref{def:F2:alg:step:findpos}.
On the other hand,
in the case where $r_1(P)=f_1(Q)$ and $r_2(P)> f_2(Q)$,
there exists $j$ such that $\numof{P_{f_1(Q),j}}>1$.
Since $r_2(P)> f_2(Q)$,
it follows from Lemma \ref{lemma:bumpingroute1} 
that
$\max(P_{f_1(Q),f_1'(Q)})<y$.
Since $\numof{P_{f_1(Q),j'}}=1$ for all $j'>f_1'(Q)$,
we can find $j$ in Step \ref{def:F2:alg:step:findpos}.
\end{remark}

\begin{example}
Consider $P$ and $Q$ in Examples \ref{example:P} and \ref{example:Q}.
Since $r_1(P)=f_1(Q)=2$ and
$r_2(P)=5>4=f_2(Q)$,
the pair $(P,Q)$ is in ${\hat{\RRR}}_\lambda$.
To calculate $\hat{\varphi}(P,Q)$,
we remove the corner $(f_2(Q),f_2'(Q))=(4,4)$
of $P$
by the reverse row insertion until $(f_1(Q)+1)$-th row.
Then we obtain
\begin{align*}
\Tableau[Y]{1&12&2&34&4&6 \\
 23&3&34&5&6&8\\3&4&5&6&\textbf{8}\\4&6&7&\tf \\ 6 }
\end{align*}
and the alphabet $7$ bumped out from $(f_1(Q)+1)$-th row.
Moreover, by inserting the alphabet $7$
into  $f_1(Q)$-th row,
we obtain
\begin{align*}
\hat{\varphi}(P,Q)=
\Tableau[Y]{1&12&2&34&4&6 \\ 23&3&34&5&\tf 6\textbf{7}&8\\3&4&5&6&\textbf{8}\\4&6&7\\ 6 }.
\end{align*}
\end{example}

\begin{definition}
For $(P,Q)\in {\hat{\RRR}}_\lambda$,
we define $\hat{\Phi}(P,Q)$ by
$\hat{\Phi}(P,Q)=(\hat{\varphi}(P,Q),F(Q))$.

\end{definition}

\section{Main results}
\label{sec:main}


For  a nonnegative integer $w$,
we define  ${\check{\RRR}}(w)$, ${\hat{\RRR}}(w)$, $\RRR^+(w)$ and $\RRR(w)$ by
\begin{align*}
{\check{\RRR}}(w)=\bigcup_{\lambda\in\YY\colon \lambda_1=w} {\check{\RRR}}_\lambda,\\ 
 {\hat{\RRR}}(w)=\bigcup_{\lambda\in\YY\colon \lambda_1=w} {\hat{\RRR}}_\lambda, \\
 \RRR^+(w)=\bigcup_{\lambda\in\YY\colon \lambda_1=w} \RRR^+_\lambda,\\  
 \RRR(w)=\bigcup_{\lambda\in\YY\colon \lambda_1=w} \RRR_\lambda.
\end{align*}
Then we have the following theorems.
\begin{theorem}[Main theorem]
\label{thm:main:bij}
For $(P,Q)\in {\check{\RRR}}(w)$,
define $\check\Phi_w(P,Q)$ to be $\check{\Phi}(P,Q)$.
Then $\check\Phi_w$
is a bijection from ${\check{\RRR}}(w)$ to ${\hat{\RRR}}(w)$
satisfying
\begin{align*}
 \mu(P)&=\mu(\hat P),&
 \nu(Q)&=\nu(\hat Q),&
 \varepsilon(P)-1&=\varepsilon(\hat P),&
 \delta(Q)+1&=\delta(\hat Q)
\end{align*}
for $(\hat P,\hat Q)=\check{\Phi}(P,Q)$.
The inverse of $\check\Phi_w$ is defined by
$\check\Phi_w^{-1}(P,Q)=\hat\Phi(P,Q)$.
\end{theorem}

We can obtain the finite sum Cauchy identity
for Grothendieck polynomials
 as a corollary to the main results.
Let $(P,Q)\in{\check{\RRR}}_\lambda$ and $(\hat P,\hat Q)\in {\hat{\RRR}}_{\hat\lambda}$
satisfy $(\hat P,\hat Q)=\check{\Phi}(P,Q)$.
If follows from the equations
that 
\begin{align*}
(-\beta)^{\varepsilon(P)}x^{\mu(P)}\cdot
 \beta^{\delta(Q)}y^{\nu(Q)}=-1\cdot(-\beta)^{\varepsilon(\hat P)}x^{\mu(\hat P)}\cdot
 \beta^{\delta(\hat Q)}y^{\nu(\hat Q)}.
\end{align*}
Since $\RRR^+_\lambda={\check{\RRR}}_\lambda\disjointunion {\hat{\RRR}}_\lambda$,
we have the following equation
\begin{align*}
&\sum_{(P,Q)\in {\check{\RRR}}(w)}
(-\beta)^{\varepsilon(P)}x^{\mu(P)}\cdot \beta^{\delta(Q)}y^{\nu(Q)}\\
=
-
&\sum_{(\hat P,\hat Q)\in {\hat{\RRR}}_{w}}
(-\beta)^{\varepsilon(\hat P)}x^{\mu(\hat P)}\cdot \beta^{\delta(\hat Q)}y^{\nu(\hat Q)},
\end{align*}
which implies
\begin{align*}
\sum_{(P,Q)\in \RRR^+(w)}
(-\beta)^{\varepsilon(P)}x^{\mu(P)}\cdot \beta^{\delta(Q)}y^{\nu(Q)}&=0.
\end{align*}
Since $(P,Q) \in \RRR^0_\lambda$
satisfies $\varepsilon(P)=0$ and $\delta(Q)=0$,
we can identified
$\RRR^0_\lambda$
with $\SSS_\lambda\times \SSS_\lambda$.
Therefore it follows that
\begin{align*}
 \sum_{\substack{\lambda\in\YY,\\ \lambda_1=w}}
 G_\lambda(\beta,x)\cdot g_\lambda(\beta,y)
&= 
\sum_{(P,Q)\in \RRR(w)}
(-\beta)^{\varepsilon(P)}x^{\mu(P)}\cdot \beta^{\delta(Q)}y^{\nu(Q)}\\
&= 
\sum_{(P,Q)\in \RRR^0(w)}
(-\beta)^{\varepsilon(P)}x^{\mu(P)}\cdot \beta^{\delta(Q)}y^{\nu(Q)}\\
&= 
\sum_{(P,Q)\in \RRR^0(w)}
x^{\mu(P)}\cdot y^{\nu(Q)}
=
\sum_{\substack{\lambda\in\YY,\\ \lambda_1=w}}
s_\lambda(x)\cdot s_\lambda(y).
\end{align*}

\section{Proof}
\label{sec:proof}
Here we show the main theorem.
First 
we recall a lemma for a bumping route.
See e.g. Fulton \cite[p.\ 9. the Row Bumping Lemma]{MR1464693} for the detail.
\begin{lemma}
 \label{lemma:bumpingroute2}
Let $T$ be an integer-valued column-strict tableau of shape $\lambda$.
First insert an alphabet $x$ into $T$ by the row bumping algorithm.
Let $T'$ be the resulting tableau, and $(i',j')$ a  new box.
Next insert an alphabet $y$ into $T'$ 
from by the row bumping algorithm.
Let $(i'',j'')$ a  new box.
\begin{enumerate}
 \item If $x\leq y$, then $i'\geq i''$ and $j' < j''$.
 \item If $x> y$, then $i'<i''$ and $j'\geq j''$.
\end{enumerate}
\end{lemma}
This lemma implies following:
\begin{lemma}
 \label{lemma:bumpingroute1}
Let $T$ be an integer-valued column-strict tableau of shape $\lambda$.
First insert an alphabet $x$ into $T$ from the $i$-th row 
by the row bumping algorithm.
Let $(i',j')$ be the new box of the resulting tableaux.
Next remove a corner $(i'',j'')$ of the original tableau $T$,
 insert the alphabet $T_{i'',j''}$ into it from the $(i''-1)$-th row
by the reverse bumping algorithm,
repeat the reverse bumping algorithm until an alphabet is bumped out
 from the $i$-th row.
Let $y$ be the alphabet bumped out from the $i$-th row.
If $i''<i'$, then $x<y$.
\end{lemma}
\begin{proof}
Let $T'$ be the tableaux obtained from $T$ by inserting the alphabet $x$
 by the row bumping algorithm,
and $T''$ the tableaux obtained from $T$ by removing the alphabet $y$
by the reverse bumping algorithm.
If we consider only rows from the $i$-th to $i'$-th row, then 
we can regard the tableaux $T''$, $T$ and $T'$
as resulting tableaux of two successive row bumping.
Hence we have the lemma by Lemma \ref{lemma:bumpingroute2}.
\end{proof}

We also show lemmas for a sliding route.
The following lemma follows from the definition of the sliding
and the reverse sliding.
\begin{lemma}
 \label{lemma:jdt1}
 Let $T$ be an integer-valued column-strict tableau.
The following are equivalent:
\begin{enumerate}
 \item If we put a null box at $(i,j)$ of $T$
       and slide it by the jeu de taquin,
       then the null box moves to the box $(i+1,j)$.        
 \item If we put a null box at $(i+1,j+1)$ of $T$
       and slide it by the reverse jeu de taquin,
       then the null box moves to the box $(i,j+1)$.        
\end{enumerate} 
\end{lemma}
This lemma implies the following lemmas for a sliding route.
\begin{lemma}
 \label{lemma:slidingroute1}
Let $T$ be an integer-valued column-strict tableau of shape $\lambda$.
First put a null box at $(i,j)$ in $T$,
and slide the null box to outside by the jeu de taquin.
Let $(i',j')$ be the null box after sliding.
Next put a null box at a cocorner $(i'',j'')$ of the original tableau $T$,
and slide the null box to inside by the reverse jeu de taquin
until the null box moves into the $j$-th row.
Let $(i'',j)$ be the null box after sliding.
If $j''\leq j'$, then $i''>i'$.
\end{lemma}
\begin{proof}
The box $(i',j')$ is a corner of $\lambda$
and $(i'',j'')$ is a cocorner of $\lambda$.
Hence, if $j''\leq j'$, then $i'<i''$.
Therefore, applying Lemma \ref{lemma:jdt1},
we have the lemma.
\end{proof}

\begin{lemma}
 \label{lemma:slidingroute2}
Let $T$ be an integer-valued column-strict tableau of shape $\lambda$.
First put a null box at $(i,j)$ in $T$,
and slide the null box to outside by the jeu de taquin.
Let $T'$ be a resulting tableau, and $(i_1,j_1)$ be the null box.
Next put a null box at $(i,j')$ in $T'$,
and slide the null box to outside by the jeu de taquin.
Let $T''$ be a resulting tableau, and $(i_2,j_2)$ be the null box.
If $j'<j$, then
$i_1 \leq i_2$ and $j_1 > j_2$.
\end{lemma}
\begin{proof}
We can obtain $T$
from $T'$ by putting a null box at $(i_1,j_1)$ 
and slide the null box by the reverse jeu de taquin.
Hence we have the lemma by Lemma \ref{lemma:slidingroute1}.
\end{proof}

Next 
we show some lemma for algorithms in Section \ref{subsec:alg:tab}.
Consider Algorithm \ref{algorithm:R(P)}.
Since the row insertion of Algorithm \ref{algorithm:R(P)}
starts from the $(r_1(P)+1)$-th row,
we have the following:
\begin{lemma}
\label{lemma:R(P):shape}
For $P\in\PPP_\lambda^+$,
then $\lambda^{(1)}_1=\lambda_1$,
where $\lambda^{(1)}$ is the shape of $R(P)$,
\end{lemma}
Algorithm \ref{algorithm:R(P)}
does not change boxes containing more than one alphabet
except the box $(r_1(P),r_1'(P))$.
Hence we have the following:
\begin{lemma}
\label{lemma:R(P)}
For $P\in\PPP_\lambda^+$, $r_1(R(P))\leq r_1(P)$.
Moreover, if $r_1(P)=r_1(R(P))$, then 
$r'_1(R(P))\leq r'_1(P)$.
\end{lemma}
Since all alphabets in $P$ are preserved,
we have the following:
\begin{lemma}
\label{lemma:mu}
For $P\in\PPP^+_{\lambda}$, $\mu(P)=\mu(R(P))$.
\end{lemma}
Since Algorithm \ref{algorithm:R(P)} adds only one box,
 Equation \eqref{eq:sum:mu} implies
 the following:
\begin{lemma}
\label{lemma:epsilon}
For $P\in\PPP^+_{\lambda}$, $ \varepsilon(P)=\varepsilon(R(P))+1$.
\end{lemma}

Consider Algorithm \ref{algorithm:F}.
The null box $(f_1(Q),f_1'(Q))$ always move to $(f_1(Q)+1,f_1'(Q))$ at 
the first step of the jeu de taquin.
Hence we have the following:
\begin{lemma}
\label{lemma:F(Q):shape}
For $Q\in\QQQ_\lambda^+$, 
 $\lambda^{(1)}_1=\lambda_1$,
where $\lambda^{(1)}$ is the shape of $F(Q)$.
\end{lemma}
Algorithm \ref{algorithm:F} changes only the boxes in
\begin{align*}
\Set{(i,j)\in\lambda| \text{$i\geq f_1(Q)$ and $j\geq f_1'(Q)$}}.
\end{align*}
Moreover the jeu de taquin preserves column-strictness.
Hence we have the following:
\begin{lemma}
\label{lemma:F(Q)}
For $Q\in\QQQ_\lambda^+$, we have $f_1(F(Q))\leq f_1(Q)$. Moreover,
if $f_1(F(Q)) =f_1(Q)$,
then  $f'_1(F(Q)) < f'_1(Q)$,
\end{lemma}

Next 
we show some lemma for algorithms in Section \ref{subsec:alg:pair}.
Consider Algorithm \ref{algorithm:checkvarphi}.
\begin{lemma}
\label{lemma:varphi:f}
For $(P,Q)\in\check{\RRR}_\lambda$,
we have
\begin{align*}
f_1(\check{\varphi}(P,Q)) &=r_3(P,Q)=r_1(P),\\
f_1'(\check{\varphi}(P,Q)) &=r_3'(P,Q).
\end{align*}
\end{lemma}
\begin{proof}
It follows from the definition of Algorithm \ref{algorithm:checkvarphi}
that $r_3(P,Q)=r_1(P)$.
Let $\hat Q=\check{\varphi}(P,Q)$.
Since the jeu de taquin preserves column-strictness,
we have 
\begin{align*}
&\Set{(i,j)|\text{$\hat Q_{i,j}=\hat Q_{i+1,j}$ for some $j$}} \\
=&
\Set{(i,j)|\text{$Q_{i,j}=Q_{i+1,j}$ for some $j$}} 
\cup \Set{(r_3(P,Q),r_3'(P,Q))}.
\end{align*}
Since $(P,Q)\in\check{\RRR}_\lambda$, we have
\begin{enumerate}
\item $r_1(P)>f_1(Q)$; or
\item $r_1(P)=f_1(Q)$ and $r_2(P)\leq f_2(Q)$.
\end{enumerate}
First assume that $r_1(P)>f_1(Q)$.
Since $r_3(P,Q)=r_1(P)$, it follows that $r_3(P,Q)>f_1(Q)$,
which implies 
$f_1(\hat Q)=r_3(P,Q)$ and $f_1'(\hat Q)=r_3'(P,Q)$.
Next assume that $r_1(P)=f_1(Q)$ and that $r_2(P)\leq f_2(Q)$.
In this case, we have $f_1(Q)=r_1(P)=r_3(P,Q)$,
which implies $f_1(\hat Q)=r_3(P,Q)$.
Moreover, since $r_2(P)\leq f_2(Q)$,
it follows from Lemma \ref{lemma:slidingroute1} 
that  $r_3'(P,Q)>f_1'(Q)$.
Hence we have $f_1'(\hat Q)=r_3'(P,Q)$.
\end{proof}
Since the jeu de taquin is the inverse of the reverse jeu de taquin,
we have the following:
\begin{lemma}
\label{lemma:varphi:f2r2}
For $(P,Q)\in \check\RRR_\lambda$, we have
\begin{align*}
f_2(\check{\varphi}(P,Q)) &=r_2(P)\\
f_2'(\check{\varphi}(P,Q)) &=r'_2(P).
\end{align*} 
\end{lemma}

We obtain $\check\varphi(P,Q)$ from $Q$ by the jeu de taquin.
Since the jeu de taquin preserves column-strictness,
we have the following:
\begin{lemma}
\label{lemma:nu}
For $(P,Q)\in {\check{\RRR}}_{\lambda}$,
$ \nu(Q)=\nu(\check\varphi(P,Q))$.
\end{lemma}
Since Algorithm \ref{algorithm:checkvarphi} adds only one new box,
Equation \eqref{eq:nu} implies the following:
\begin{lemma}
\label{lemma:delta}
For $(P,Q)\in {\check{\RRR}}_{\lambda}$,
$ \delta(Q)=\delta(\check\varphi(P,Q))-1$.
\end{lemma}

Consider Algorithm \ref{algorithm:hatvarphi}.
\begin{lemma}
\label{lemma:hatphi:r}
For $(P,Q)\in\hat\RRR_\lambda$, we have
\begin{align*}
r_1(\hat{\varphi}(P,Q)) &=f_3(P,Q)=f_1(Q),\\
r_1'(\hat{\varphi}(P,Q)) &=f_3'(P,Q).
\end{align*}
\end{lemma}
\begin{proof}
It follows from the definition of Algorithm \ref{algorithm:hatvarphi}
that $f_3(P,Q)=f_1(Q)$.
Let $\check P =\hat{\varphi}(P,Q)$.
It follows that
\begin{align*}
  \Set{(i,j)| \text{ $\numof{\check P_{i,j}}>1$ for some $j$}}\\
=  \Set{(i,j)| \text{ $\numof{ P_{i,j}}>1$ for some $j$}}
\cup \Set{(f_3(P,Q),f_3'(P,Q))}.
\end{align*}
Since $(P,Q)\in\hat\RRR_\lambda$, we have
\begin{enumerate}
\item $r_1(P)<f_1(Q)$; or
\item $r_1(P)=f_1(Q)$ and $r_2(P)> f_2(Q)$.
\end{enumerate}
If $r_1(P)<f_1(Q)$, then 
it is easy to show that
$r_1(\check P)=f_3(P,Q)$ and that $r_1'(\check P)=f_3'(P,Q)$.
Assume that $r_1(P)=f_1(Q)$ and that $r_2(P)> f_2(Q)$.
In this case, we have $r_1(P)=f_1(Q)=f_3(P,Q)$.
Hence $r_1(\hat P)=f_3(P,Q)$.
Moreover, since $r_2(P)> f_2(Q)$,
it follows from Lemma \ref{lemma:bumpingroute1} 
that
$\max(P_{f_1(Q),f_1'(Q)})<y$,
which implies $f_3'(P,Q) \geq f_1'(Q)$.
Hence we have
$r_1'(\check P)=f_3'(P,Q)$.
\end{proof}
Since the row insertion is the inverse of the reverse row insertion,
we have the following:
\begin{lemma}
\label{lemma:hatvarphi:f2r2}
For $(P,Q)\in\hat\RRR_\lambda$, we have
\begin{align*}
r_2(\hat{\varphi}(P,Q)) &=f_2(Q)\\
r_2'(\hat{\varphi}(P,Q)) &=f'_2(Q).
\end{align*}
 \end{lemma}

Finally we consider $\check{\Phi}$ and  $\hat{\Phi}$.
\begin{lemma}
\label{lemma:fromchecktohat}
Let $(P,Q)\in {\check{\RRR}}_\lambda$,  
If $(\hat P,\hat Q)=\check{\Phi}(P,Q)$,
then  $(\hat P,\hat Q)$ is in $\hat{\RRR}_{\hat\lambda}$,
where $\hat\lambda$ is the shape of $\hat P$.
\end{lemma}
\begin{proof}
Since $\hat Q=\check{\varphi}(P,Q)$,
we have $f_1(\hat Q)=r_1(P)$
by Lemma \ref{lemma:varphi:f}.
Since $\hat P=R(P)$,
we have 
$r_1(\hat P) \leq r_1(P)$
 by Lemma \ref{lemma:R(P)}.
Hence $r_1(\hat P) \leq f_1(\hat Q)$.
If $r_1(\hat P) < f_1(\hat Q)$,
then $(\hat P,\hat Q)\in \hat{\RRR}_{\hat\lambda}$.
Assume that $r_1(\hat P) = f_1(\hat Q)$.
Let
\begin{align*}
 x&=\max (P_{r_1(P),r'_1(P)})\\
 \hat x&=\max(\hat P_{r_1(\hat P),r'_1(\hat P)}). 
\end{align*}
Since $f_1(\hat Q)=r_1(P)$,
we have $r_1(\hat P) = r_1(P)$.
By Lemma \ref{lemma:R(P)},
we have $r'_1(\hat P) \leq r'_1(P)$.
Hence $\hat x \leq x$.
Consider only the rows strictly below the $r_1(P)$-th row.
Since 
we obtain $\hat P$ from $P$ by the row insertion with $x$,
it follows from Lemma \ref{lemma:bumpingroute2}
that $r_2(P)<r_2(\hat P)$.
By Lemma \ref{lemma:varphi:f2r2},
we have $r_2(P)=f_2(\hat Q)$.
Hence $f_2(\hat Q)<r_2(\hat P)$ and $(\hat P,\hat Q)\in \hat{\RRR}_{\hat\lambda}$.
\end{proof}

\begin{lemma}
\label{lemma:fromhattocheck}
Let $(P,Q)\in {\hat{\RRR}}_\lambda$,  
If $(\check P,\check Q)=\hat{\Phi}(P,Q)$,
then  $(\check P,\check Q)$ is in $\check{\RRR}_{\check\lambda}$,
where $\check\lambda$ is the shape of $\check P$.
\end{lemma}
\begin{proof}
Since $\check P=\hat{\varphi}(P,Q)$,
we have $f_1(Q)=r_1(\check P)$
by Lemma \ref{lemma:hatphi:r}.
Since $\check Q=F(Q)$,
we have 
$f_1(\check Q) \leq f_1(Q)$
 by Lemma \ref{lemma:F(Q)}.
Hence $f_1(\check Q) \leq r_1(\check P)$.
If $f_1(\check Q) < r_1(\check P)$,
then $(\check P,\check Q)\in \check{\RRR}_{\check\lambda}$.
Assume that $f_1(\check Q) = r_1(\check P)$.
Since $f_1(Q)=r_1(\check P)$,
we have $f_1(\check P) = f_1(P)$.
 By Lemma \ref{lemma:F(Q)}.
we have $f'_1(\check Q) < f'_1(Q)$.
Hence, by Lemma \ref{lemma:slidingroute2},
we have $f_2(Q)\leq f_2(\check Q)$.
By Lemma \ref{lemma:hatvarphi:f2r2},
we have $f_2(Q)=r_2(\check P)$.
Hence $r_2(\check P)\leq f_2(\check Q)$ and $(\check P,\check Q)\in \check{\RRR}_{\check\lambda}$.
\end{proof}

\begin{lemma}
\label{lemma:bij1}
Let $(P,Q)\in {\check{\RRR}}_{\lambda}$ and $(\hat P,\hat Q)\in {\hat{\RRR}}_{\hat\lambda}$.
If $(\hat P,\hat Q)=\check{\Phi}(P,Q)$,
then  $\hat\Phi(\hat P,\hat Q)=(P,Q)$.
\end{lemma}
\begin{proof}
 Let $(\hat P,\hat Q)=\check{\Phi}(P,Q)$.
Since $\hat Q = \check\varphi(P,Q)$,
we have
\begin{align*}
 f_1(\hat Q)&=r_3(P,Q)\\
 f'_1(\hat Q)&=r'_3(P,Q)
\end{align*}
by Lemma \ref{lemma:varphi:f}.
Hence the box $(f_1(\hat Q),f'_1(\hat Q))$
is the box where 
the reverse jeu de taquin stops in Algorithm \ref{algorithm:checkvarphi}.
Since the jeu de taquin is 
 the inverse of  the reverse jeu de taquin,
we have $F(\hat Q)=Q$.
We also have
\begin{align*}
 f_2(\hat Q)&=r_2(P)\\
 f'_2(\hat Q)&=r'_2(P)
\end{align*}
by Lemma \ref{lemma:varphi:f2r2}.
Hence the box $(f_2(\hat Q), f'_2(\hat Q))$
is the new box added by  Algorithm \ref{algorithm:R(P)}.
Since the reverse row insertion is 
 the inverse of  the row insertion,
we have $\hat\varphi(\hat P,\hat Q)=P$.
\end{proof}

\begin{lemma}
\label{lemma:bij2}
Let $(\check P,\check Q)\in {\check{\RRR}}_{\check\lambda}$ and $(P,Q)\in {\hat{\RRR}}_{\lambda}$.
If $(\check P,\check Q)=\hat{\Phi}(P,Q)$,
then  $\check\Phi(\check P,\check Q)=(P,Q)$.
\end{lemma}
\begin{proof}
 Let $(\check P,\check Q)=\hat{\Phi}(P,Q)$.
Since $\check P = \hat\varphi(P,Q)$,
we have
\begin{align*}
 r_1(\check P)&=f_3(P,Q)\\
 r'_1(\check P)&=f'_3(P,Q)
\end{align*}
by Lemma \ref{lemma:hatphi:r}.
Hence the box 
$ (r_1(\check Q), r'_1(\check Q))$
is the box where 
Algorithm \ref{algorithm:hatvarphi} appends an alphabet.
Since the row insertion is
 the inverse of  the row insertion,
we have $R(\check P)=P$.
We also have
\begin{align*}
 r_2(\check P)&=f_2(Q)\\
 r'_2(\check P)&=f'_2(Q)
\end{align*}
by Lemma \ref{lemma:hatvarphi:f2r2}.
Hence the box $(r_2(\check P), r'_2(\check P))$
is the box where 
the jeu de taquin stops in Algorithm \ref{algorithm:F}.
Since the reverse jeu de taquin is 
 the inverse of  jeu de taquin,
we have $\check\varphi(\check P,\check Q)=Q$.
\end{proof}

Since we have
Lemmas \ref{lemma:R(P):shape}
and \ref{lemma:fromchecktohat},
we can define
the map $\check\Phi_w$ from 
${\check{\RRR}}(w)$ to ${\hat{\RRR}}(w)$
by $\check\Phi_w(P,Q)=\check{\Phi}(P,Q)$.
Since we also have
Lemmas \ref{lemma:F(Q):shape}
and \ref{lemma:fromhattocheck},
we can define
the map $\hat\Phi_w$ from 
${\hat{\RRR}}(w)$ to ${\check{\RRR}}(w)$
by $\hat\Phi_w(P,Q)=\hat{\Phi}(P,Q)$.
It follows from  Lemmas  \ref{lemma:bij1} and \ref{lemma:bij2}
that $\check\Phi_w$ is the inverse of $\hat\Phi_w$.
Hence $\check\Phi_w$ and $\hat\Phi_w$ are bijections.
Since we also have Lemmas 
\ref{lemma:mu}, \ref{lemma:epsilon},
\ref{lemma:nu} and \ref{lemma:delta},
we obtain Theorem \ref{thm:main:bij}.

\section{The case of  Young diagrams with one column}
\label{sec:finalrmk}
Here we consider only Young diagrams with one column.
In this case, 
we describe our bijection explicitly.
Let $(P,Q)\in \RRR^+_{(1^l)}$.
In this case,
$(P,Q)\in{\check{\RRR}}_{(1^l)}$ if and only if $r_1(P)>f_1(Q)$.
For $(P,Q)$ with $r_1(P)>f_1(Q)$,
it follows that $(P',Q')=\check{\Phi}(P,Q)$,
where $(P',Q')\in \RRR^+_{(1^{l+1})}$ 
is a pair of tableaux obtained by expand their $r_1(P)$-th rows, 
i.e., a pair of tableaux defined by
\begin{align*}
 P'_{i,1}&=\begin{cases}
P_{i,1} &(\text{if $1\leq i<r_1(P)$})\\
P_{r_1(P),1} \setminus  \Set{\max(P_{r_1(P),1})}&(\text{if $i=r_1(P)$})\\
\Set{\max(P_{r_1(P),1})} &(\text{if $i=r_1(P)+1$})\\
P_{i-1,1} &(\text{if $r_1(P)+1<i\leq l+1$})
	  \end{cases}
\\
 Q'_{i,1}&=\begin{cases}
Q_{i,1} &(\text{if $1\leq i\leq r_1(P)$})\\
Q_{i-1,1} &(\text{if $r_1(P)+1\leq i\leq l+1$}).
	  \end{cases}
\end{align*}
For $(P,Q)$ with $r_1(P) \leq f_1(Q)$,
it follows that $(P',Q')=\hat{\Phi}(P,Q)$,
where 
$(P',Q')\in \RRR^+_{(1^{l+1})}$ 
is the pair of tableaux obtained by folding their $f_1(Q)$-th rows, 
i.e.,
the pair  of tableaux defined by
\begin{align*}
 P'_{i,1}&=\begin{cases}
P_{i,1} &(\text{if $1\leq i<f_1(Q)$})\\
P_{f_1(Q),1} \cup P_{f_1(Q)+1,1}&(\text{if $i=f_1(Q)$})\\
P_{i+1,1} &(\text{if $r_1(Q)+1<i\leq l-1$})
	  \end{cases}
\\
 Q'_{i,1}&=\begin{cases}
Q_{i,1} &(\text{if $1\leq i<f_1(Q)$})\\
Q_{f_1(Q),1}=Q_{f_1(Q)+1,1} &(\text{if $i=f_1(Q)$})\\
Q_{i+1,1} &(\text{if $f_1(Q)< i\leq l-1$}).
	  \end{cases}
\end{align*}

\bibliographystyle{amsplain-url} 
\bibliography{x2.bib} 

\providecommand{\bysame}{\leavevmode\hbox to3em{\hrulefill}\thinspace}
\providecommand{\MR}{\relax\ifhmode\unskip\space\fi MR }
\providecommand{\MRhref}[2]{%
  \href{http://www.ams.org/mathscinet-getitem?mr=#1}{#2}
}
\providecommand{\href}[2]{#2}
\begin{thebibliography}{1}

\bibitem{MR3007174}
Jason Bandlow and Jennifer Morse, \emph{Combinatorial expansions in
  {$K$}-theoretic bases}, Electron. J. Combin. \textbf{19} (2012), no.~4, Paper
  39, 27. \MR{3007174}

\bibitem{MR1932326}
Anders~Skovsted Buch, \emph{Grothendieck classes of quiver varieties}, Duke
  Math. J. \textbf{115} (2002), no.~1, 75--103, URL
  \url{http://dx.doi.org/10.1215/S0012-7094-02-11513-0}. \MR{1932326
  (2003m:14018)}

\bibitem{MR1946917}
\bysame, \emph{A {L}ittlewood-{R}ichardson rule for the {$K$}-theory of
  {G}rassmannians}, Acta Math. \textbf{189} (2002), no.~1, 37--78, URL
  \url{http://dx.doi.org/10.1007/BF02392644}. \MR{1946917 (2003j:14062)}

\bibitem{MR1394950}
Sergey Fomin and Anatol~N. Kirillov, \emph{The {Y}ang-{B}axter equation,
  symmetric functions, and {S}chubert polynomials}, Proceedings of the 5th
  {C}onference on {F}ormal {P}ower {S}eries and {A}lgebraic {C}ombinatorics
  ({F}lorence, 1993), vol. 153, 1996, pp.~123--143, URL
  \url{http://dx.doi.org/10.1016/0012-365X(95)00132-G}. \MR{1394950}

\bibitem{MR1464693}
William Fulton, \emph{Young tableaux}, London Mathematical Society Student
  Texts, vol.~35, Cambridge University Press, Cambridge, 1997, With
  applications to representation theory and geometry. \MR{1464693 (99f:05119)}

\bibitem{MR2377012}
Thomas Lam and Pavlo Pylyavskyy, \emph{Combinatorial {H}opf algebras and
  {$K$}-homology of {G}rassmannians}, Int. Math. Res. Not. IMRN (2007), no.~24,
  Art. ID rnm125, 48, URL \url{http://dx.doi.org/10.1093/imrn/rnm125}.
  \MR{2377012 (2009i:16066)}

\bibitem{MR3249830}
Alain Lascoux and Hiroshi Naruse, \emph{Finite sum {C}auchy identity for dual
  {G}rothendieck polynomials}, Proc. Japan Acad. Ser. A Math. Sci. \textbf{90}
  (2014), no.~7, 87--91, URL \url{http://dx.doi.org/10.3792/pjaa.90.87}.
  \MR{3249830}

\bibitem{MR686357}
Alain Lascoux and Marcel-Paul Sch{\"u}tzenberger, \emph{Structure de {H}opf de
  l'anneau de cohomologie et de l'anneau de {G}rothendieck d'une vari\'et\'e de
  drapeaux}, C. R. Acad. Sci. Paris S\'er. I Math. \textbf{295} (1982), no.~11,
  629--633. \MR{686357}

\end{thebibliography}
\end{document}